\documentclass[a4paper,12pt]{amsproc}

\usepackage[utf8]{inputenc}

\usepackage{amsmath}
\usepackage{amssymb}
\usepackage{mathrsfs}
\usepackage{enumerate}
\usepackage{geometry}
\usepackage[dvips]{graphicx}

\newtheorem{theorem}{Theorem}
\newtheorem{lemma}[theorem]{Lemma} %%[section]
\newtheorem{proposition}[theorem]{Proposition} %%[section]
\newtheorem{cor}[theorem]{Corollary} %%[section]
\newtheorem{fact}[theorem]{Fact} %%[section]
\newtheorem{remark}[theorem]{Remark} %%[section]
\newtheorem{definition}[theorem]{Definition}
\newtheorem{problem}[theorem]{Question}

\def\bbn{\mathbb{N}}

\def\bbs{\mathbb{S}}
\def\w{\omega}

\def\ct{\mathbb{T}}
\newcommand{\mbb}{\mathbb} 
\newcommand{\baire}{\omega^\omega}
\newcommand{\mc}{\mathcal}
\newcommand{\mf}{\mathfrak}

\def\ca{\mathcal{A}}

\def\cd{\mathcal{D}}
\def\cf{\mathcal{F}}

\def\ci{\mathcal{I}}

\def\cp{\mathcal{P}}

\def\b{\mathfrak{b}}
\def\c{\mathfrak{c}}
\def\d{\mathfrak{d}}
\DeclareMathOperator{\Sacks}{\bbs} %{SacksTrees}

\DeclareMathOperator{\ZFC}{ZFC}

\DeclareMathOperator{\add}{add}
\DeclareMathOperator{\non}{non}
\DeclareMathOperator{\cov}{cov}
\DeclareMathOperator{\cof}{cof}

\def\ed{{ e.d.\,}}
\def\med{{ m.e.d.\,}}

\def\restricted{\upharpoonright}

\def\then{\longrightarrow}

\def\then{\longrightarrow}

\newcommand{\bez}{\backslash}
\newcommand{\se}{\subseteq}
\newcommand{\sen}{\subsetneq}
\newcommand{\es}{\supseteq}

\newcommand{\0}{\emptyset}

\newcommand{\foralmostall}{\forall^\infty}
\newcommand{\existsinfty}{\exists^\infty}

\def\<{\langle}
\def\>{\rangle}

\def\Repicky{{Repick{\'y} }}

\title[Nonmeasurable sets and unions]{Nonmeasurable sets and unions with respect to tree ideals}
\author{Marcin Michalski, Robert Rałowski and Szymon Żeberski}

\address{ Department of Computer Science, Faculty of Fundamental Problems of Technology, Wroc\l aw University of Science and Technology, Wybrzeże Wyspiańskiego 27, 50-370 Wroc\l aw, Poland}

\thanks{\hspace*{-0.77cm} AMS Classification: Primary 03E17, 03E50, 03E75; Secondary 28A99
The work has been partially financed by grant S50129/K1102 (0401/0086/16) from the Faculty of Fundamental Problems of Technology of Wrocław University of Technology.}

\email[Marcin Michalski]{marcin.k.michalski@pwr.edu.pl}
\email[Robert Rałowski]{robert.ralowski@pwr.edu.pl}
\email[Szymon Żeberski]{szymon.zeberski@pwr.edu.pl}

\begin{document}

\begin{abstract} In this paper we consider a notion of nonmeasurablity with respect to Marczewski and Marczewski-like tree ideals $s_0$, $m_0$, $l_0$, and $cl_0$. We show that there exists a subset $A$ of the Baire space $\omega^\omega$ which is $s$-, $l$-, and $m$-nonmeasurable, that forms dominating m.e.d. family. We introduce and investigate a notion of $\mbb{T}$-Bernstein sets - sets that intersect but does not containt any body of a tree from a given family of trees $\mbb{T}$. We also acquire some results on $\mc{I}$-Luzin sets, namely we prove that there are no $m_0$-, $l_0$-, and $cl_0$-Luzin sets and that if $\mf{c}$ is a regular cardinal, then the algebraic sum (considered on the real line $\mbb{R}$) of a generalized Luzin set and a generalized Sierpiński set belongs to $s_0, m_0$, $l_0$ and $cl_0$.

\end{abstract}

\maketitle

%\section{Definitions}
\section{Introduction and preliminaries}
We will use standard set-theoretic notation following e.g. \cite{Jech}.  
For a set $X$, $P(X)$ denotes the power set of $X$ and $|X|$ denotes the cardinality of $X$. 
If $\kappa$ is a cardinal number then we denote:
\begin{itemize}
\item $[X]^\kappa \hspace{0.23cm}=\{A\subseteq X:\ |A|=\kappa\}$,
\item $[X]^{<\kappa}=\{A\subseteq X:\ |A|<\kappa\}$,
\item $[X]^{\le\kappa}=\{A\subseteq X:\ |A|\le\kappa\}$.
\end{itemize}
Let $X$ be an uncountable Polish space and $\ci\subseteq P(X)$ be a  $\sigma$-ideal. Let us recall some cardinal coefficients from Cichoń's Diagram:
\begin{itemize}
\item $\add(\ci)=\min \{ |\ca|:\; \ca\subseteq \ci\land \bigcup\ca\notin \ci \}$,
\item $\non(\ci)=\min \{ |A|:\; A\subseteq X\land A\notin \ci\}$,
\item $\cov(\ci)=\min \{ |\ca|:\; \ca\subseteq \ci\land \bigcup\ca = X\}$,
\item $\cof(\ci)=\min \{ |\ca|:\; \ca\subseteq \ci\land (\forall A\in \ci)(\exists B\in\ca) (A\subseteq B)\}$,
\item $\b=\min\{|\cf|: \cf\se\omega^\omega \land (\forall x\in\baire)(\exists f\in\cf)(\existsinfty n)(x(n)<f(n))\}$,
\item $\d=\min\{|\cf|: \cf\se\omega^\omega \land (\forall x\in\baire)(\exists f\in\cf)(\foralmostall n)(x(n)<f(n))\}$.
%\item for a fixed family of perfect subsets $\cp \subseteq Perf(X)$ let\\ $\cov_h(\ci)=\min \{ |\ca|:\; \ca\subseteq \ci\land (\exists P\in \cp)\; P\subseteq \bigcup\ca\}$.
\end{itemize}
%Let us recall the definition of a bounding number.
%$$ 
% \b = \min \{| \cb |:\ \cb \subseteq \w^\w \land (\forall x \in \w^\w) (\exists y \in \cb) \; \neg (s \le^* x) \} 
% $$ 

%Notice that for this ideal we have $\cov(s_0)=\cov_h(s_0)$ and this cardinal is the same for all uncountable Polish spaces. 
%To see this use the fact that in any uncountable Polish space there is a disjoint maximal antichain $\ca$ (of cardinality $\c$) consisting of Cantor perfect sets. 
%From this it follows that $B\in s_0$ if and only if $(\forall A\in\ca)\; B\cap A\in s_0$.

%For fixed $\sigma$-ideal $\ci$ with Borel base we say that a subset $A\subseteq X$ of Polish space $X$ is measurable with respect to $\ci$ iff $A$ belongs to $\sigma$-algebra $Bor[\ci]$ generated by Borel subsets of $X$ and $\sigma$-ideal $\ci$. 

%In the first part of this paper we consider subsets connected to $\sigma$-ideal without Borel base generated by trees. We are interested in measurability connected to Laver trees and Miller trees and it's interplay with m.e.d. families.

%In the second part we investigate subsets connected to $\sigma$-ideals with Borel base.
%We discuss the difference between measurability and complete nonmeasurability of unions of small sets.

%%%%%%%%%%%%%%%%%%%%%%%%%%%%%%%%%%%%%%%%%% MAD families %%%%%%%%%%%%%%%%%%%%%%%%%%%%%%%%%%%%%%%
%\section{Maximal eventually different families and their non-measurability} %$s$, $l$ and $m$-measurability
We call $\b$ a bounding number and $\d$ a dominating number. A family $\cf\se\omega^\omega$ is dominating, if $\cf$ has a property described in the definition of domintaing number (it doesn't have to be of minimal cardinality).
\\
We say that $T$ is a tree on a set $A$ if $T\se A^{<\omega}$ and whenever $\tau\in T$ then $\tau\restricted n\in T$ for each natural $n$.
\begin{definition}
Let $T$ be a tree on a set $A$. Then
\begin{itemize}
\item for each $t\in T$ $succ(t)=\{a\in A: t^\frown a\in T\}$;
\item $split(T)=\{t\in T: |succ(t)|\geq 2\}$;
\item $\omega$-$split(T)=\{t\in T: |succ(t)|=\aleph_0\}$;
\item for $s\in T$ $Succ_T(s)=\{t\in split(T): s\sen t, (\forall t'\in T)(s\sen t' \sen t \then t'\notin split(T) )\}$;
\item for $s\in T$ $\omega$-$Succ_T(s)=\{t\in \omega$-$split(T): s\sen t, (\forall t'\in T)(s\sen t' \sen t \then t'\notin \omega$-$split(T) )\}$;
\item $stem(T)\in T$ is a node $\tau$ such that for each $s\subsetneq\tau$ $|succ(s)|=1$ and $|succ(\tau)|>1$.
\end{itemize}
\end{definition}

Let us now recall definitions of families of trees.
%%%%%%%%%%%%%% Prefect trees %%%%%%%%%%%%
\begin{definition}
A tree $T$ on $\omega$ is called
\begin{itemize}
\item Sacks tree or perfect tree, denoted by $T\in\mbb{S}$, if for each node $s\in T$ there is $t\in T$ such that $s\se t$ and $|succ(t)|\geq 2$;
\item Miller tree or superperfect tree, denoted by $T\in\mbb{M}$, if for each node $s\in T$ exists $t\in T$ such that $s\se t$ and $|succ(t)|=\aleph_0$;
\item Laver tree, denoted by $T\in\mbb{L}$, if for each node $t\es stem(T)$ we have $|succ(t)|=\aleph_0$;
\item complete Laver tree, denoted by $T\in\mbb{CL}$, if $T$ is Laver and $stem(T)=\0$;
\item Hechler tree, denoted by $T\in\mbb{H}$, if  for each node $t\es stem(T)$ we have that a set $\{n\in\omega: t^\frown n\notin T\}$ is finite;
\item complete Hechler, denoted by $T\in\mbb{CH}$ tree, if $T$ is Hechler and $stem(T)=\0$.
\end{itemize}
\end{definition}
The notion of complete Laver trees was defined and investigated in \cite{R1}, although Miller in \cite{Miller1} defines Laver trees \textit{de facto} as complete Laver trees and Hechler trees as complete Hechler trees.

For every tree $T\subseteq \omega^{<\omega}$ let $[T]$ be the set of all infinite branches of $T$, i.e.
$$
[T]=\{ x\in \omega^\omega:\; (\forall n\in\omega)\; x\restricted n\in T\}.
$$

\begin{definition}[Tree ideal] Let $\ct$ be a family of trees. We say that $A\in P(\omega^\omega)$ is in $t_0$ iff
$$
(\forall P\in \ct)(\exists Q\in \ct)\; Q\subseteq P\land [Q]\cap A=\emptyset.
$$
\end{definition}

\begin{definition}[$t$-measurability] Let $\ct$ be a family of trees. We say that $A\in P(\omega^\omega)$ is $t$-measurable iff
$$
(\forall P\in \ct)(\exists Q\in \ct)\; Q\subseteq P\land ([Q]\subseteq A\lor [Q]\cap A=\emptyset).
$$
\end{definition}

$s_0$ tree ideal is simply a classic Marczewski ideal (see \cite{Marczewski}).

It is well known due to Judah, Miller, Shelah (see \cite{JMS}) and \Repicky (see \cite{Rep}) that $add(s_0)\le cov(s_0)\le cof(\c)\le non(s_0)=\c<cof(s_0)\le 2^\c$. Moreover, in \cite{BKW} Brendle, Khomskii and Wohofsky have shown that also $\c<cof(m_0)$ and $\mf{c}<cof(l_0)$. Clearly $\omega_1\le \add(l_0)\le \cov(l_0)\le \c$ holds. In \cite{Goldstern}, Goldstern, \Repicky\!\!, Shelah and Spinas showed that it is relatively consistent with $\ZFC$ that $\add(l_0)<\cov(l_0)$.

%Next, we say that tree a $T\subseteq \omega^{<\omega}$ is a {\bf complete Laver tree}, iff, every node $t\in T$ is infinitely splitting.

%The set of all complete Laver trees is denoted by the $\Superperfect$. Moreover, recalling the definition of the ideal $sp_0$, we have
%\begin{definition}[ideal $sp_0$] We say that $A\in \cp(\omega^\omega)$ is in $sp_0$ iff
%$$
%(\forall T\in \Superperfect)(\exists Q\in \Superperfect)\; Q\subseteq T\land [Q]\cap A=\emptyset.
%$$
%\end{definition}
%Using the natural isomorphism of the perfect set $P=[T]$, where $T$ is Laver tree, with $\omega^\omega$, we have $cov_h(l_0)=cov(l_0)$.

%\begin{definition} Let $\cc\in\{ s,l,m\}$ then we say that a family $\ca$ of sets is {\bf $\cc$-summable} iff for every subfamily $\cb\subseteq \ca$ the union $\bigcup \cb$ is an $\ci$-measurable set.
%\end{definition}

%Let us recall the definition of almost disjoint family. Any family of sets $\ca\subseteq [\omega]^\omega$ is an {\bf \ad--family} on $\omega$ if
%$$
%(\forall a,b\in\ca)\; a\ne b\then a\cap b\in [\omega]^{<\omega}.
%$$

Let us notice that the families $s_0, l_0, m_0$ form  $\sigma$-ideals. On the other hand $cl_0$ is not a $\sigma$-ideal.
To see that it is enough to consider sets of the form $C_n=\{x\in\omega^\omega:\ x(0)=n \}.$ Then $C_n\in cl_0$ for each $n$, but $\bigcup_n C_n=\omega^\omega$. Using the fact that $s_0$ is a $\sigma$-ideal we may give another proof of the following well known result.
\begin{proposition}[Essentially a joke]
$cf(\mf{c})>\aleph_0$.
\end{proposition}
\begin{proof}
Suppose that $cf(\mf{c})=\aleph_0$ and let $\mbb{R}=\bigcup_{n\in\omega}A_n$, $|A_n|<\mf{c}$ for each n$\in\omega$. Sets of cardinality lesser than $\mf{c}$ belong to $s_0$, so $\mbb{R}=\bigcup_{n\in\omega}A_n\in s_0$, a contradiction.
\end{proof}
\section{Tree ideals and measurability}
In \cite{Brendle} the following result was obtained.
\begin{theorem}[Brendle]
		If $i_0, j_0\in \{s_0, l_0, m_0\}$ and $i_0\neq j_0$ then $i_0\not\subseteq j_0$.	
\end{theorem}	
First we will compare the ideal $cl_0$ with ideals $s_0, m_0, l_0$. 
\begin{fact}
		$cl_0\not\subseteq (l_0\cup m_0\cup s_0)$.
\end{fact}
\begin{proof}
	 To show the assertion let us take $C_0=\{x\in\omega^\omega:\ x(0)=0\}$. By
	 $\mathbb{CL}\subseteq\mathbb{L}\subseteq\mathbb{M}\subseteq\mathbb{S}$, $[C_0]\notin l_0\cup m_0\cup s_0$. On the other hand $[C_0]\in cl_0$, which finishes the proof.
\end{proof}	
	
\begin{theorem}\label{cl0 a inne}
	      The following statements are true:
		\begin{enumerate}[(i)]
			\item $m_0\not\subseteq cl_0$. 
			\item $s_0\not\subseteq cl_0$.
		\end{enumerate}
\end{theorem}
\begin{proof}
	To prove that $m_0\setminus cl_0\neq\emptyset$ we will slightly modify the proof of Theorem 2.1 from \cite{Brendle}. We will use the notions of apple trees and pear trees.
	\\
	First, let us recall that each Miller tree contains an apple tree and each apple tree is a special kind of a Miller tree (apple trees forms a dense subfamily in all Miller trees).
	\\
	Second, each complete Laver tree $C$ contains a pear tree $P_C$. A pear tree is not a complete Laver tree, it is only a special kind of Sacks tree. Pear trees $P_C$ have the following property: for every apple tree $A$ and pear tree $P_C$ $|[A]\cap [P_C]|\le 1$.
	\\
	Let us now enumerate all apple trees $\{A_\alpha:\ \alpha<{\mathfrak c}\}$ and all complete Laver trees $\{C_\alpha:\ \alpha<\mathfrak c\}$. 
	Having the above two propositions we can proceed by induction and construct a sequence $(x_\alpha)_{\alpha<\mf{c}}$ such that for every $\alpha<\mathfrak c$:
	\[
	x_\alpha\in [P_{C_\alpha}]\setminus\bigcup_{\beta<\alpha}[A_\beta]
	\]
	Finally, we set $X=\{x_\alpha:\ \alpha<\mathfrak c\}.$ Let us notice that $X\in m_0\setminus cl_0$, which finishes the first part of the proof.
	
	To prove that $s_0\setminus cl_0\neq\emptyset$ we use slight modification of the proof of Theorem 2.2 from \cite{Brendle}, which fits a similar pattern from the first case. 
	  
	%We will use special kind of trees defined by Brendle in \cite{Brendle}.
\end{proof}	

\begin{problem}\label{Question cl0 a l0}
	Is it true that $l_0\not\subseteq cl_0$?
\end{problem}	

As a consequence we obtain the following result.	
\begin{cor}
	   The following statements are true:
	\begin{enumerate}[(i)]
		\item There exists a $cl$-nonmeasurable set which is  $m$-measurable.
		\item There exists a $cl$-nonmeasurable set which is  $s$-measurable.
	\end{enumerate}	
\end{cor}
Let us introduce a notion of $\mbb{T}$-Bersntein sets.
\begin{definition}
Let $\mbb{T}$ a family of trees. We say that a set $B$ is an $\mbb{T}$-Bernstein set if for every $T\in\mbb{T}$ $B\cap [T]\neq\0$ and $B\bez [T]\neq \0$.
\end{definition}
Observe that a classic Bernstein set is an $\mbb{S}$-Bernstein set. If $\mbb{T}\se \mbb{T}'$ are families of trees, then $\mbb{T}'$-Bersntein sets are $\mbb{T}$-Bernstein sets. No $\mbb{T}$-Bernstein set is in $t_0$ (or $t$-measurable), and if $\mbb{T}\se \mbb{T}'$ then $\mbb{T}'$-Bernstein sets don't belong to $t_0$. Also note that if $\mbb{T}\sen \mbb{T}'$ then a $\mbb{T}$-Bernstein set may be not a $\mbb{T}'$-Bernstein set (e.g. one may fix a tree from $\mbb{T}'\bez\mbb{T}$ which body will be always omitted). The following theorem slightly generalizes Theorems 2.1 and 2.2 from \cite{Brendle}.
\begin{theorem}
The following statements are true:
\begin{enumerate}[(i)]
\item There exists an $\mbb{L}$-Bernstein set which belongs to $m_0$.
\item There exists an $\mbb{M}$-Bernstein set which belongs to $s_0$.
\end{enumerate}
\end{theorem}
\begin{proof}
As in in the proof of Theorem \ref{cl0 a inne} we will use notions established in \cite{Brendle}. To prove (i) let us enumerate all Laver trees $\{L_\alpha : \alpha<\mf{c}\}$ and all apple trees $\{A_\alpha: \alpha<\mf{c}\}$. Let us construct two sequences: $(b_\alpha)_{\alpha<\mf{c}}$ and $(x_\alpha)_{\alpha<\mf{c}}$ such that for each $\alpha<\mf{c}$:
\begin{align*}
b_\alpha\in & [L_\alpha]\bez(\bigcup_{\beta<\alpha}[A_\beta]\cup \{x_\xi: \xi<\alpha\}),
\\
x_\alpha\in &[L_\alpha]\bez(\{b_\beta : \beta\leq \alpha\}\cup\{x_\beta : \beta<\alpha\}).
\end{align*}
It can be done, since for each Laver tree $L_\alpha$ there is a pear tree $P_{L_\alpha}$ for which $|[P_{L_\alpha}]\cap [A]|\leq 1$ for every apple tree $A$, so the set $[L_\alpha]\bez(\bigcup_{\beta<\alpha}[A_\beta]\cup \{x_\xi: \xi<\alpha\})$ is nonempty at each step $\alpha$. Then $B=\{b_\alpha: \alpha<\mf{c}\}$ is the desired set.
\\
To prove (ii) we use a similar modification of Theorem 2.2 from \cite{Brendle}.
\end{proof}
Analogously to the Question \ref{Question cl0 a l0} we may ask the following question.
\begin{problem}
Is there a $\mbb{CL}$-Bernstein set which belongs to $l_0$?
\end{problem}
Let us invoke a theorem by Miller from \cite{Miller1}.
\begin{theorem}[Miller]\label{Laver lub Hechler}
Let $A\in\Sigma^1_1$. Either $A$ contains body of some complete Laver tree or $A^c$ contains a body of some complete Hechler tree.
\end{theorem}

\begin{theorem}\label{intersection of Borel and tree ideal}
The following is true:
\begin{enumerate}[(i)]
\item $\mc{B}\cap s_0$ is an ideal of Borel sets that don't contain a perfect subset (so it's an ideal of countable sets).
\item $\mc{B}\cap m_0$ is an ideal of Borel sets which don't contain a body of any Miller tree.
\item $\mc{B}\cap l_0$ is an ideal of Borel sets that don't contain a body of any Laver tree.
\end{enumerate}
\end{theorem}
\begin{proof}
(i) is evident.
\\
(ii) follows by the fact that any analytic set is either $\sigma$ - bounded or contains a superperfect set. If a Borel set contains a superfect set then clearly it's not in $m_0$. On the other hand, if for some Miller tree $T$ and $\sigma$ - bounded Borel a set $B$ $[T]\bez B$ didn't contain a superperfect set, then $[T]$ would be $\sigma$ - bounded too. A contradiction.
\\
(iii): If a Borel set $B$ contains a body of some Laver tree, then clearly $B\notin l_0$. If it doesn't contain a Laver tree, but there is a Laver $L$ for which each body of Laver subtree of $L$ has a nonempty intersection with $B$, then let us trim $B$ and $L$ in the following way:
\begin{eqnarray*}
B'&=&\{x\in\baire: stem(L)^\frown x\in B\},
\\
L'&=&\{x\in\baire:  stem(L)^\frown x\in L\}.
\end{eqnarray*}
A function $f: \baire\rightarrow \baire$ given by the formula $f(x)=stem(L)^\frown x$ is continuous. Clearly, $B'=f^{-1}[B]$, so $B'$ is Borel, and $[L']=f^{-1}[[L]]$ is a body of a complete Laver tree $L'$. $B'$ still doesn't contain a body of any Laver tree, so by Theorem \ref{Laver lub Hechler} there is a Hechler tree $H$ which body is contained in $B'^c$. $H\cap L'$ contains (in fact - is) a Laver tree, body of which $B'$ should intersect - a contradiction.
\end{proof}
\begin{definition}
We say that a set $A$ is $\mc{I}$-nonmeasurable if $A\notin \sigma(\mc{B}\cup \mc{I})$. $A$ is completely $\mc{I}$-nonmeasurable if $A\cap B$ is $\mc{I}$-nonmeasurable for each Borel set $B\notin \mc{I}$, or equivalently - $A$ intersects each, but doesn't contain any, Borel $\mc{I}$-positive set.
\end{definition}
\begin{cor}
Let $(\mbb{T}, t_0)\in\{(\mbb{S}, s_0), (\mbb{M}, m_0), (\mbb{L}, l_0)\}$. Then a set $B$ is a $\mbb{T}$-Bernstein iff it is completely $t_0\cap\mc{B}$-nonmeasurable. 
\end{cor}
\begin{proof}
By Theorem \ref{intersection of Borel and tree ideal} a set $A$ is $t_0\cap\mc{B}$-positive Borel set if and only if it contains a body of some tree from $\mbb{T}$, so a set $B$ is $\mbb{T}$-Bernstein if and only if it intersects each each, but does not contain any, Borel set containing a body of a tree from $\mbb{T}$.
\end{proof}

\section{$\mc{I}$-Luzin sets and algebraic properties}
	
Let us recall the notion of $\mathcal I$-Luzin sets. Let $X$ be a Polish space and $\mc{I}$ be an ideal.
\begin{definition}
		We say that a set $L$ is an $\mathcal I$-Luzin set if $(\forall A\in\mathcal I)(|A\cap L|<|L|)$. 
\end{definition}
For classic ideals of Lebesgue measure zero sets $\mc{N}$ and meager sets $\mc{M}$ we will call $\mc{M}$-Luzin sets generalized Luzin sets and $\mc{N}$-Luzin sets generalized Sierpiński sets.
 
In \cite{Wohofsky} the following result was proven.
\begin{theorem}[Wohofsky]
		There is no $s_0$-Luzin set.
\end{theorem}
	
We will show that similar results can be obtained for other tree ideals.	

\begin{theorem}
	    The following statements are true.
		\begin{enumerate}[(i)]
			\item There is no $l_0$-Luzin set.
			\item There is no $cl_0$-Luzin set.
			\item There is no $m_0$-Luzin set.
		\end{enumerate}
\end{theorem}
	\begin{proof}
		Let us consider  $l_0$ case. We will prove that for every set  $X$ of cardinality $\mathfrak c$ there exists a set $A\subseteq X$ such that  $A\in l_0$ and $|A|=\mathfrak c$. Indeed, let us assume that $X\notin l_0$. Then there exists  $L\in\mathbb{L}$ such that for every $L'\subseteq L$, $L'\in\mathbb{L}$ we have $|[L]\cap X|=\mathfrak c$.
			Let us now fix a maximal antichain $\{L_\alpha:\alpha<\mathfrak c\}$ of Laver trees contained in $L$ such that $|[L_\alpha]\cap X|=\mathfrak c$.
			Let us construct a sequence $(a_\alpha)_{\alpha<\mf{c}}$ such that for each $\alpha<\mf{c}$:			
			\[
			a_\alpha \in X\setminus \bigcup_{\xi<\alpha}[L_\alpha].
			\]			 
			Then $A=\{a_\alpha:\ \alpha<\mathfrak c\}$ is the set. Proofs of the other cases are almost identical.
\end{proof}
Now we will consider $\mc{I}$-Luzin sets in a context of algebraic properties and tree ideals. We will work on the real line $\mbb{R}$ with addition. Since $\mbb{R}$ is $\sigma$-compact, it does not contain even superperfect sets. We will tweak the definition a bit by saying that $A\se \mbb{R}$ belongs to $t_0$ if $h^{-1}[A]$ belongs to $t_0$ in $\baire$, where $h$ is a homeomorphism between $\baire$ and a subspace of  irrational numbers (see \cite{KyWe} for a similar modification in the case of $2^\omega$). Having this in mind we will usually mean by $[\tau]$, $\tau\in\omega^{<\omega}$, an open interval of rational endpoints on $\mbb{R}$.

Before we proceed let us define a non-standard kind of fusion of Miller and Laver trees, that we will use later. Let $T$ be a Miller tree. Let $\tau_\0\in \omega$-$split(T)$ and let $T_0$ be any Miller subtree of $T$ such that $\tau_\0$ remains an infinitely splitting node in $T_0$. Suppose we have a Miller subtree $T_n$ and a set of nodes $B_n=\{\tau_\sigma: \sigma\in n^{\leq n}\}$ such that
\begin{enumerate}[(i)]
\item $\tau_\sigma\in\omega$-$split(T_n)$ for every $\sigma\in n^{\leq n}$;
\item $\tau_{\sigma^\frown k}\es\tau_\sigma$ for every $k<n$ and $\sigma\in n^{<n}$;
\item $\tau_{\sigma^\frown k}\cap\tau_{\sigma^\frown j}=\tau_\sigma$ for every $\sigma\in n^{<n}$ and distinct $k, j<n$.
\end{enumerate}
We extend the set of nodes $B_n$ to $B_{n+1}=\{\tau_\sigma: \sigma\in (n+1)^{\leq n+1}\}$ in a way that preserves above conditions, so we gonna have n+1 levels of infinitely splitting nodes with fixed n+1 splits. The only $\sigma\in (n+1)^0$ is $\0$, and $\tau_\0$ is an old node. It is $\omega$-$splitting$ in $T_n$ and $T_n$ is a Miller tree, so we may find $\tau_n\es \tau_\0$, which is $\omega$-$splitting$ and $\tau_n\cap \tau_j=\tau_\0$ for $j<n$. If we already have $\tau_\sigma$'s with desired properties for $\sigma\in (n+1)^{\leq k}$, $k<n+1$, then for $\tau_\sigma$, $\sigma\in n^k$ (old node), we add $\tau_{\sigma^\frown n}$ such that conditions (i) - (iii) are still met. %$\tau_{\sigma^\frown n}\es \tau_\sigma$, $\tau_{\sigma^\frown n}\in\omega$-$split(T_n)$ and $\tau_{\sigma^\frown n}\cap \tau_{\sigma^\frown j}=\tau_\sigma$ for $j<n$% 
For a new node $\tau_\sigma$, $\sigma\in (n+1)^k\bez n^k$, we find $\tau_{\sigma^\frown j}$ for each $j<n+1$ such that conditions (i) - (iii) are satisfied too. Then let $T_{n+1}$ be any Miller subtree of $T_n$ for which nodes from $B_{n+1}$ are still infinitely splitting.
\\
We will call a sequence of trees $(T_n)_{n\in\omega}$ (or, interchangeably, their bodies $[T_n]$) derived that way a \textit{Miller fusion sequence}.
\\
Similarly we define a \textit{Laver fusion sequence}. The only difference would be that if $\tau_\sigma\se\tau_{\sigma^\frown k}$, then actually $\tau_{\sigma^\frown k}={\tau_{\sigma}}^\frown j$ for some $j\in\omega$.
\begin{proposition}
For every Miller (resp. Laver) fusion sequence $(T_n)_{n\in\omega}$ a set $\bigcap_{n\in\omega}T_n$ is a Miller (resp. Laver) tree.
\end{proposition}

\begin{lemma}\label{Fusion lemma for intervals}
For every sequence of intervals $(I_n)_{n\in\omega}$ and a Miller (resp. Laver) tree $T$ there is a Miller (resp. Laver) fusion sequence $(T_n)_{n\in\omega}$ such that for all $n>0$:
\[
\lambda([T_n]+I_n)< (1+\Sigma_{k=0}^{n-1}(n-1)^k)\lambda(I_n).
\]
\end{lemma}
\begin{proof}
Let us focus on a little more complicated "Miller" case. Let $I_0$ be an interval, $\lambda(I_0)=\epsilon_0$, $T$ a Miller tree. We proceed by induction on $n$. Let $\tau_\0\in\omega$-$split(T)$ such that $\lambda([\tau_\0])<\epsilon_0$. Then $\lambda([\tau_\0]+I_0)=\lambda([\tau_\0])+\lambda(I_0)<2\epsilon_0$. Let $T_0$ be Miller subtree of $T$ such that $\tau_\0=stem(T_0)$ and $\tau_\0\in\omega$-$split(T_0)$. Clearly, we have $\lambda([T_0]+I_0)<2\epsilon_0$. 
\\
Now assume that we have a tree $T_n$ that is an element of the emerging Miller fusion sequence, and associated with it set $B_n$ of fixed nodes satisfying conditions (i) - (iii). Let $\lambda(I_{n+1})=\epsilon_{n+1}$. Let us denote for each $\sigma\in\omega^{<\omega}$ and interval $I_\sigma$ a set
\[
N(I_\sigma)=\{{\tau_\sigma}^\frown k\in T_n: [{\tau_\sigma}^\frown k]\se I_\sigma\land(\forall j<n)(\tau_{\sigma^\frown j}\not\es {\tau_\sigma}^\frown k)\}.
\]
At each level $k<n$ for every $\sigma\in n^k$ let $I_\sigma$ be an interval with $\lambda(I_\sigma)<\frac{\epsilon_{n+1}}{(n+1)^{n}}$ such that a set $N(I_\sigma)$ is infinite and choose $\tau_{\sigma^\frown n}\in\omega$-$split(T_n)$ such that $\tau_{\sigma^\frown n}\es {\tau_\sigma}^\frown l$ for some ${\tau_\sigma}^\frown l\in N(I_\sigma)$.
At the level $n$ let us fix an intervals $I_\sigma$, $\lambda(I_\sigma)<\frac{\epsilon_{n+1}}{(n+1)^{n}}$, for $\sigma\in n^n$ such that sets $N(I_\sigma)$ are infinite and pick $\tau_{\sigma^\frown 0}, \tau_{\sigma^\frown 1}, ..., \tau_{\sigma^\frown n}$ which are extensions of some nodes ${\tau_\sigma}^\frown k_0, {\tau_\sigma}^\frown k_1, ..., {\tau_\sigma}^\frown k_n\in N(I_\sigma)$ respectively. Finally we pick remaining nodes to complete a set $B_{n+1}$ in the gist of our definition of Miller fusion sequence however we like. We take as $T_{n+1}$ any Miller subtree of $T_n$ for which nodes from $B_{n+1}$ are infinitely splitting and which body is covered by intervals $I_\sigma,\, \sigma\in n^{\leq n}$ (which is possible by infiniteness of each $N(I_\sigma)$).
\\
Let us approximate $\lambda([T_{n+1}]+I_{n+1})$:
\begin{align*}
\lambda([T_{n+1}]+I_{n+1})&\leq \lambda(\bigcup\{I_\sigma+I_{n+1}: \sigma\in n^{\leq n}\}\leq\Sigma_{\sigma\in n^{\leq n}}(\lambda(I_\sigma)+\lambda(I_{n+1}))<
\\
&< \Sigma_{\sigma\in n^{\leq n}}(\frac{\epsilon_{n+1}}{(n+1)^{n}}+\epsilon_{n+1}),
\end{align*}
and since the count of intervals $I_\sigma$ is  $|n^{\leq n}|=\Sigma_{k=0}^{n}n^k\leq (n+1)^{n}$, we have:
\begin{align*}
\lambda([T_{n+1}]+I_{n+1})&\leq \Sigma_{k=0}^{n}n^k(\frac{\epsilon_{n+1}}{(n+1)^{n}}+\epsilon_{n+1})\leq (n+1)^{n}\frac{\epsilon_{n+1}}{(n+1)^{n}}+\Sigma_{k=0}^{n}n^k\epsilon_{n+1}=
\\
&=\epsilon_{n+1}+\Sigma_{k=0}^{n}n^k\epsilon_{n+1}=(1+\Sigma_{k=0}^{n}n^k)\epsilon_{n+1}. 
\end{align*}
\end{proof}
\begin{remark}\label{Remark on freezing stem}
In the above Lemma in the case of a Laver tree we may demand that $stem(T)=stem(\bigcap_{n\in\omega}T_n)$, if $stem(T)$ is nonempty.
\end{remark}
\begin{proof}
The major difference is at the first step of the induction. Instead of picking a suitable "far enough" node $\tau_\0\in T$ such that $\lambda([\tau_\0]+I_0)<2\lambda(I_0)$, we already restrict the choice of nodes at the stem level by picking an interval $I_\0$ of measure $\lambda(I_\0)<\lambda(I_0)$ such that a set
\[
N(I_\0)=\{stem(T)^\frown k\in T: [stem(T)^\frown k]\se I_\0\}
\]
is infinite. It can be done since $stem(T)\neq\0$, so all clopens $[stem(T)^\frown k]$, $k\in\omega$, are contained in an interval. We take a Laver subtree $T_0$ of $T$ for which $[T]\se I_\0$ and $stem(T)=stem(T_0)$ (so all nodes extending $stem(T_0)$ come from $I_\0$). Then we continue analogously to the proof of the Lemma \ref{Fusion lemma for intervals}.
\end{proof}
\begin{lemma}\label{Miller fusion lemma for G delta}
There exists a dense $G_\delta$ set $G$ such that for each Miller (resp. Laver or complete Laver) tree $T$ there exists a Miller (resp. Laver or complete Laver) subtree $T'\se T$ such that $G+[T']\in\mc{N}$.
\end{lemma}
\begin{proof}
Let $D=\{d_n: n\in\omega\}$ be a countable dense set, $G=\bigcap_{n\in\omega}\bigcup_{k>n}I_k$, where $I_k$ is an interval with center $d_k$ and $\lambda(I_k)<\frac{1}{(k)^{k-1}2^k}$. Proofs are almost identical in cases of Miller and Laver trees so let $T$ be a Miller tree. By the Lemma \ref{Fusion lemma for intervals} there is a Miller fusion sequence $(T_n)_{n\in\omega}$ such that
\[
\lambda([T_n]+I_n)<(1+\Sigma_{k=0}^{n-1}(n-1)^k)\lambda(I_n)\leq n^{n-1}\frac{1}{n^{n-1}2^n}=\frac{1}{2^n}.
\]
$T'=\bigcap_{n\in\omega}T_n$ is a Miller tree containing all $T_n$'s, so we may replace $[T_n]$ with $[T']$ in the above formula and it still holds. Then for fixed $n\in\omega$:
\[
\lambda(\bigcup_{k>n}I_k+[T'])= \lambda(\bigcup_{k>n}([T']+I_k))\leq\Sigma_{k>n}\lambda([T']+I_k)\leq\Sigma_{k>n}\frac{1}{2^k}=\frac{1}{2^n},
\]
so, given that $[T']+\bigcap_{n\in\omega}\bigcup_{k>n}I_k\se \bigcap_{n\in\omega}\bigcup_{k>n}([T']+I_k)$, we have:
\[
\lambda(G+[T'])\leq \lambda(\bigcap_{n\in\omega}\bigcup_{k>n}([T']+I_k))\leq \lim_{n\rightarrow\infty}\frac{1}{2^n}=0.
\]
In the case of a complete Laver tree $T$ let us observe that $T=\bigcup_{n\in\omega}T_n$, where $T_n=\{\sigma\in T: (n)\se\sigma\lor\sigma\se (n)\}$ is a Laver tree with a nonempty stem. Let us notice that $[T]=\bigcup_{n\in\omega}[T_n]$. By the Lemma \ref{Fusion lemma for intervals}, Remark \ref{Remark on freezing stem}, and using the first part of the proof we find for each (nonempty) $T_n$ a Laver subtree $T_n'$ which shares the stem with $T_n$ and for which we have:
\[
[T_n']+G\in\mc{N}.
\]
Then $T'=\bigcup_{n\in\omega}T_n'$ is a complete Laver subtree of $T$ and:
\[
[T']+G=[\bigcup_{n\in\omega}T_n']+G=\bigcup_{n\in\omega}[T_n']+G=\bigcup_{n\in\omega}([T_n']+G)\in\mc{N}
\]
as a countable union of null sets.
\end{proof}
Before we proceed to the main theorem of this section let us recall a generalized version of Rothberger's theorem (see \cite{Roth}).
\begin{theorem}(Essentially Rothberger)\label{Rothberger theorem}
Assume that generalized Luzin set $L$ and generalized Sierpiński set $S$ exist. Then, if $\kappa=max\{|L|, |S|\}$ is a regular cardinal, $|L|=|S|=\kappa$.
\end{theorem}
\begin{proof}
Assume that $\kappa=|L|>|S|$ and $\kappa$ is a regular cardinal. Let $M$ be a meager set of full measure (the Marczewski decomposition). Then
\[
\kappa=|L\cap\mbb{R}|=|L\cap (M+S)|=|\bigcup_{s\in S}(L\cap(M+s))|<\kappa,
\]
by regularity of $\kappa$. In the case of $\kappa=|S|>|L|$ the proof is almost the same. 
\end{proof}
The following theorem extends the result obtained in \cite{MZ}.
\begin{theorem}
Let $\c$ be a regular cardinal and $t_0\in\{s_0, m_0, l_0, cl_0\}$. Then for every generalized Luzin set $L$ and generalized Sierpiński set $S$ we have $L+S\in t_0$.
\end{theorem}
\begin{proof}
Let $L$ and $S$ be a generalized Luzin set and generalized Sierpiński set respectively. If $|L|<\c$ and $|S|<\c$, then $L+S\in t_0$, since every set of cardinality less than $\c$ belongs to $t_0$. So, without a loss of generality (Theorem \ref{Rothberger theorem}), let us assume that $|L|=|S|=\c$.
\\
We will proceed with the proof in the case $t_0=m_0$, the other cases are almost identical. Let $T$ be a Miller tree. By the virtue of Lemma \ref{Miller fusion lemma for G delta} let $G$ be a dense $G_\delta$ set and $T'\se T$ a Miller tree such that $[T']+G\in \mc{N}$. Let $A=-G$ and $B=([T']+G)^c$. Then $[T']\se(A+B)^c$. We will show that there is a Miller tree $T''\se T'$ which body is contained in $(L+S)^c$. We have:
\begin{align*}
L+S&=((L\cap A)\cup(L\cap A^{c}))+((S\cap B)\cup(S\cap B^{c}))
\\
& = ((L\cap A)+(S\cap B))\cup((L\cap A)+(S\cap B^{c}))\cup
\\
&\cup((L\cap A^{c})+(S\cap B))\cup((L\cap A^{c})+(S\cap B^{c})).
\end{align*}
$(L\cap A)+(S\cap B)\subseteq A+B$ and sets $(L\cap A)+(S\cap B^{c})$, $(L\cap A^{c})+(S\cap B)$ and $(L\cap A^{c})+(S\cap B^{c})$ are generalized Luzin, generalized Sierpi\'nski and of size less than $\mf{c}$, so their intersection with $[T']$ has a cardinality less than $\mf{c}$. It follows that indeed there exists a Miller tree $T''\subseteq T'$ such that $(L+S)\cap [T'']=\emptyset$ and therefore $L+S$ belongs to $m_0$.
\end{proof}
Let us remark that the assumption that $\mf{c}$ is regular cannot be omitted due to the following result (\cite{MZ}).
\begin{theorem}
It is consistent that there exist generalized Luzin set $L$ and generalized Sierpi{\'n}ski set $S$ such that $L+S=\mbb{R}^n$, and $\mf{c}=\aleph_{\omega_1}$.
\end{theorem}
		
\section{Eventually different families and t-measurablity}
		
Two members $f,g\in\omega^\omega$ of the Baire space are \textit{eventually different} (briefly: e.d.) iff $f\cap g$ is a finite subset of $\omega\times\omega$. 
%Let us observe that an \ed family $\ca\subseteq\omega^\omega$ is an \ad family on $\omega\times\omega$. 
%For this reason we will identify the notions of eventually different family and almost disjoint \ad family. 
Maximal eventually different families with respect to inclusion are called {\it \med families}.

Every e.d. family is a meager subset of the Baire space. 
It is natural to ask whether the existence of m.e.d. families that are either $s$-measurable or $s$-nonmeasurable can be proven in $\ZFC$. It is relatively consistent with ZFC that there is a m.e.d. family $\ca$ of cardinality smaller then $\mf{c}$ (see \cite{Kunen}). In such a case $\ca\in s_0$.
On the other hand there exists a perfect \ed family and therefore not all m.e.d. families are in $s_0$. The following two theorems answer this question positively.
\begin{theorem}\label{nonmeasurable_mad} There exists an $s$-nonmeasurable  m.e.d. family in the Baire space.
\end{theorem}
\begin{proof} Let us fix a perfect tree $T\subseteq \omega^{<\omega}$ such that $[T]$ is e.d. in $\omega^\omega$. Let $\{ T_\alpha: \alpha<\c\}$ be an enumeration of $\Sacks(T)$ - a family of all perfect subtrees of $T$. 
By transfinite reccursion we define:
$$
\{ (a_\alpha,d_\alpha,x_\alpha)\in [T]\times [T]\times \omega^\omega:\alpha<\c\}
$$
such that for any $\alpha<\c$ we have:
\begin{enumerate}[\hspace{0.5cm}(1)]
 \item $a_\alpha,d_\alpha\in [T_\alpha]$,
 \item $\{ a_\xi:\xi<\alpha\}\cap \{ d_\xi:\xi<\alpha\}=\emptyset$,
 \item $\{ a_\xi:\xi<\alpha\}\cup \{ x_\xi:\xi<\alpha\}$ is \ed,
 \item $\forall^\infty n\; x_\alpha(n)=d_\alpha(n)$ but $x_\alpha\ne d_\alpha$.
\end{enumerate}
Assume that we are at the step $\alpha<\mf{c}$ of the construction and we have already defined the sequence:
$$
\{ (a_\xi,d_\xi,x_\xi)\in [T]^2\times \omega^\omega:\xi<\alpha\}.
$$
We can choose $a_\alpha,d_\alpha\in [T_\alpha]$ ($[T_\alpha]$ has cardinality $\c$) which fulfills conditions $(1),(2)$. 
Then choose any $x_\alpha\in \omega^\omega$ distinct from $d_\alpha$ but $(\forall^\infty n) d_\alpha(n)=x_\alpha(n)$. Then $x_\alpha\in\omega^\omega\setminus [T]$ and 
\[
\{ a_\xi:\xi<\alpha\}\cup \{ x_\xi:\xi<\alpha\}
\]
forms an \ed family in $\omega^\omega$. This completes the construction.\
\\
Now let us set $A_0 = \{ a_\alpha:\alpha < \c\}\cup \{ x_\alpha:\alpha<\c\}$ and let us extend it to m.e.d. family $A$. It is easy to check that $A$ is the desired $s$-nonmeasurable m.e.d. family.
\end{proof}

In \cite{R1} it was shown that if $\d = \w_1$ then there exists a $s$-nonmeasurable m.e.d. family $\ca$ and $\ca'\in [\ca]^{\w_1}$ which is dominating in $\w^\w$. 
Here $s$-nonmeasurability can be replaced by $l$-, $m$- or $cl$-nonmeasurability.

In the same paper it was proved that the following statement is relatively consistent with ZFC: "$\omega_1 < \d$ and there exists $cl$-nonmeasurable m.e.d. family $\ca$ 
and a dominating family $\ca'\subseteq \ca$ of the cardinality equal to $\d$".

The next theorem generalizes the result obtained in \cite{R1}.
\begin{theorem}\label{d_l_mad} There exists a m.e.d. family $\ca\subseteq \omega^\omega$ such that $\ca$ is not $s$-, $l$- and $m$-measurable, with a dominating subfamily $\cd\in [\ca]^{\le \d}$.
\end{theorem}
\begin{proof} By definition there is a dominating family $\cd_0\subseteq \omega^\omega$ of size $\d$. We will show that there is an a.d. dominating family $\cd$ of the same size. Let $\cp = \{ A_m\in [\omega]^\omega:\; m\in\omega \}$ be a partition of $\omega$ into infinite subsets. Let us construct a tree as follows: 
$T_{-1}=\{ \emptyset \}$, next $T_0=\{ (0,n): n\in\omega\}$. 
Now assume that we have defined $T_n$ for a fixed $n\in\omega$ and let us enumerate $T_n=\{ s_k: k\in\omega\}$ then for every $m\in\omega$ let us set $A_m=\{ k_{m, i}: i\in\omega\}$ as an increasing sequence with $i$ running through $\omega$ and $m$ fixed. Define 
$T_{n+1,m} = \{ s_m\cup \{(n+1, k_{m, i})\}:i\in\omega\}$ and then let $T_{n+1}=\bigcup_{m\in\omega} T_{n+1,m}$ and finally $T=\bigcup_{n\in\omega\cup\{ -1\}} T_n$.  
It is easy to observe that $[T]$ forms an a.d. family in $\omega^\omega$. 

Now let us define an embedding $f:\cd_0\to [T]$ as follows: pick an arbitrary element $d\in\cd_0$ which is an union $\bigcup\{ d\upharpoonright n: n\in\omega\}$ 
then assign to $d\restricted 0=\emptyset\in T_{-1}$ and to $d\restricted 1$ $t_0=d\restricted 1=\{ (0,d(0))\}$. Now let us assume that we have assigned for a fixed 
$d\restricted n$  $t_n\in T_n$ for $n\in\omega$. Then there is unique $m\in\omega$ such that $t_n\in T_{n,m}$ but $A_m=\{ k_{m, i}: i\in\omega\}$ is represented by the increasing sequence 
$(k_{m, i})_{i\in\omega}\in\omega^\omega$ then $d\restricted n+1$ is assigned to $t_{n+1} = t_n\cup \{ (n+1,w)\}$ where $w=k_{m, d(n+1)}$ which is a greater than $d(n+1)$. 
From the construction we see that $t_{n+1}\in T_{n+1}$ and for any $n\in\omega$ $t_n\subseteq t_{n+1}$. Now let $f(d)=\bigcup\{ t_n\in T_n: n\in\omega:\}\in [T]$. 
It easy to see that this ensures that $f$ is one to one mapping and for any $d\in \cd_0$ $d\le f(d)$. Now let $\cd=\{ 4f(d): d\in\cd_0\}\subseteq (4\bbn)^\omega$ 
which forms a dominating family in $\omega^\omega$ of size equal to $\d=|\cd_0|$.

Now let us choose a.d. trees $S\subseteq (4\bbn+1)^{<\omega}$, $M\subseteq (4\bbn+2)^{<\omega}$ and $L\subseteq (4\bbn+3)^{<\omega}$ where $S$ is a perfect tree, $M$ is Miller 
and $L$ is Laver.

%%%%%%%%%%%% construction of s-nonmeasurable m.e.d. family %%%%%%%%%%%
Let  us enumerate $\mbb{S}(S)=\{ S_\alpha: \alpha<\c\}$ - a family of all perfect subtrees of $S$, analogously $\mbb{M}(M)=\{ M_\alpha:\alpha<\c\}$, and $\mbb{L}(L)=\{ L_\alpha:\alpha<\c\}$. 
By transfinite reccursion let us define
\[
\{ w_\alpha \in [S]^2\times \omega^\omega\times [M]^2\times \omega^\omega\times [L]^2\times \omega^\omega:\alpha<\c\}
\]
where $w_\alpha=(a^s_\xi,d^s_\xi,x^s_\xi,a^m_\xi,d^m_\xi,x^m_\xi,a^l_\xi,d^l_\xi,x^l_\xi,)$ for any $\alpha<\c$, and such that for any $\alpha<\c$ we have:
\begin{enumerate}
 \item $a^s_\alpha,d^s_\alpha\in [S_\alpha]$,
 \item $\{ a^s_\xi:\xi<\alpha\}\cap \{ d^s_\xi:\xi<\alpha\}=\emptyset$,
 \item $\{ a^s_\xi:\xi<\alpha\}\cup \{ x^s_\xi:\xi<\alpha\}$ is e.d.,
 \item $\forall^\infty n\; x^s_\alpha(n)=d^s_\alpha(n)$ but $x^s_\alpha\ne d^s_\alpha$.
 
 \item $a^m_\alpha,d^m_\alpha\in [M_\alpha]$,
 \item $\{ a^m_\xi:\xi<\alpha\}\cap \{ d^m_\xi:\xi<\alpha\}=\emptyset$,
 \item $\{ a^m_\xi:\xi<\alpha\}\cup \{ x^m_\xi:\xi<\alpha\}$ is e.d.,
 \item $\forall^\infty n\; x^m_\alpha(n)=d^m_\alpha(n)$ but $x^m_\alpha\ne d^m_\alpha$.

 \item $a^l_\alpha,d^l_\alpha\in [L_\alpha]$,
 \item $\{ a^l_\xi:\xi<\alpha\}\cap \{ d^l_\xi:\xi<\alpha\}=\emptyset$,
 \item $\{ a^l_\xi:\xi<\alpha\}\cup \{ x^l_\xi:\xi<\alpha\}$ is e.d.,
 \item $\forall^\infty n\; x^l_\alpha(n)=d^l_\alpha(n)$ but $x^l_\alpha\ne d^l_\alpha$.
\end{enumerate}
Now assume that we are at the step $\alpha<\mf{c}$ of the construction and we have a partial sequence:
\[
\{ w_\alpha :\; \xi<\alpha\}
\]
which has a length at most $\omega\cdot |\alpha|<\c$. In the case of the perfect part we can choose in $[S_\alpha]$ (of size $\c$) $a^s_\alpha,d^s_\alpha\in [S_\alpha]$ which fulfills the first condition. 
Then choose any $x^s_\alpha\in \omega^\omega$ different than $d^s_\alpha$ but $(\forall^\infty n) d_\alpha(n)=x_\alpha(n)$ then $x^s_\alpha\in\omega^\omega\setminus [S]$ and 
\[
\{ a_\xi:\xi \le \alpha\}\cup \{ x_\xi:\xi \le \alpha\}
\]
forms an e.d. family in $\omega^\omega$. In the same way we can choose other points of our tuple for Miller and Laver trees. The construction is complete. Now let us set: 
\[
\mc{A}_s = \cd\cup \{ a^s_\alpha:\alpha < \c\}\cup \{ x^s_\alpha:\alpha<\c\},
\]
\[
\ca_m = \cd\cup \{ a^m_\alpha:\alpha < \c\}\cup \{ x^m_\alpha:\alpha<\c\}
\]
and
\[
\ca_l = \cd\cup \{ a^l_\alpha:\alpha < \c\}\cup \{ x^l_\alpha:\alpha<\c\}.
\]
Let us extend the family $\cd\cup\ca_s\cup\ca_m\cup\ca_l$ to any m.e.d. family $\ca$. It is easy to check that $\ca$ is required $s$-, $m$- 
and $l$-nonmeasurable m.e.d. family in $\omega^\omega$ with a dominating subfamily of size $\d$, which completes the proof.
\end{proof}

\end{document}